\documentclass[noinfoline]{article}

\RequirePackage[OT1]{fontenc}
\RequirePackage[
amsthm,amsmath
]{imsart}

\usepackage{graphicx}
\usepackage{latexsym,amsmath}
\usepackage{amsmath,amsthm,amscd}
\usepackage{amsfonts}
\usepackage[psamsfonts]{amssymb}
\usepackage{enumerate}
\usepackage{mathtools}

\usepackage{url}

\usepackage{hyperref}

\theoremstyle{plain} 
\newtheorem{theorem}{Theorem}
\newtheorem{corollary}[theorem]{Corollary}

{Lemma}

\newtheorem{proposition}[theorem]{Proposition}

\theoremstyle{definition} 

\theoremstyle{definition} 

\newtheorem*{ex*}{Example}
\theoremstyle{remark} 

\theoremstyle{remark} 
\newtheorem{remark}[theorem]{Remark}
\newtheorem*{remark*}{Remark}

\newcommand{\sign}{\operatorname{sign}}

\newcommand{\al}{\alpha}

\newcommand{\ga}{\gamma}

\newcommand{\be}{\beta}

\newcommand{\vpi}{\varphi}
\renewcommand{\th}{\theta}
\newcommand{\E}{\operatorname{\mathsf{E}}}

\newcommand{\CC}{\mathbb{C}}

\newcommand{\R}{\mathbb{R}}

\newcommand{\J}{\mathcal{J}}

\newcommand{\vp}{\varepsilon}

\renewcommand{\le}{\leqslant}

\renewcommand{\Re}{\operatorname{\mathsf{Re}}}
\renewcommand{\Re}{\operatorname{\mathfrak{Re}}}

\begin{document}

\begin{frontmatter}

\title{Characteristic function of the positive part of a random variable and related results, with applications
}
\runtitle{Characteristic function of the positive part of a random variable
}
\author{Iosif Pinelis}
\runauthor{Iosif Pinelis}

\address{Department of Mathematical Sciences\\
Michigan Technological University\\
Houghton, Michigan 49931, USA\\
E-mail: ipinelis@mtu.edu}

\begin{abstract}
Let $X$ be an arbitrary real-valued random variable (r.v.), with the characteristic function (c.f.) $f$.  
Integral expressions for the c.f.\ of the r.v.'s $\max(0,X)$ 
in terms of $f$ 
are given, as well as other related results. 
Applications to stock options and  random walks  are presented. 
In particular, a more explicit and compact form of Spitzer's identity is obtained. 
\end{abstract}


\end{frontmatter}

\section{Introduction}\label{intro}

The positive part, $X_+:=0\vee X=\max(0,X)$, of a random variable (r.v.) $X$ arises in various contexts, including inequalities and extremal problems in probability, statistics, operations research, and finance; see e.g.\ \cite{spitzer,T2,bent-AP,pin-hoeff-published,q-bounds-published} and further references therein. In particular, in finance $(S-K)_+$ is the value of a call option with strike price $K$ when the underlying stock price is $S$. 
Also, the absolute value of a r.v.\ is easily expressible in terms the positive-part operation: $|X|=X_++X_-$, where $X_-:=(-X)_+$. 
Motivated by different needs and using rather different methods,  
\cite{brown70} and \cite{positive} provided various expressions for the power moments $\E X_+^p$ ($p>0$) of the positive part $X_+$ in terms of the Fourier or Fourier--Laplace transform of (the distribution of) $X$. 

However, expressions for the characteristic function (c.f.) $f_{X_+}$ of the r.v.\ $X_+$ in terms of the c.f.\ $f_X$ of an $X$ appear to be absent in the existing literature. 
Here such expressions are provided, as well as related ones. 
We also include applications to stock options and random walks. 
In particular, a more explicit and compact form of Spitzer's identity is obtained. 

\section{Basic results}\label{results}

Let $X$ be any real-valued r.v., and let then $f=f_X$ denote the characteristic function (c.f.) of $X$, so that $f(t)=\E e^{itX}$ for all real $t$. 
Introduce  
\begin{align}
	(J_a f)(t):=&\frac1{2\pi i}\,\int_{-\infty}^\infty e^{-iua}f(t+u)\,\frac{du}u \label{eq:J} \\ 
	=& \frac1{4\pi i}\,\int_{-\infty}^\infty\big[e^{-iua}f(t+u)-e^{iua}f(t-u)\big]\,\frac{du}u. \label{eq:symm}
\end{align}
Here and subsequently, $a,b,s,t,u,x,\al,\be,\ga$ stand for arbitrary real numbers (unless otherwise specified) and the integral $\int_{-\infty}^\infty$ is understood in the principal-value sense, 
so that 
\begin{equation}\label{eq:J_a=}
(J_a f)(t)=\lim_{\vp\downarrow0,\,A\uparrow\infty}(J_{a;\vp,A} f)(t), 	
\end{equation}
where
\begin{equation}\label{eq:J_{a;vp,A}}
	(J_{a;\vp,A} f)(t):=\frac1{2\pi i}\,\int_{\vp,A} e^{-iua}f(t+u)\,\frac{du}u\quad\text{and}\quad
	\int_{\vp,A}:=\int_\vp^A+\int_{-A}^{-\vp}.  
\end{equation}
The equality in \eqref{eq:symm} follows by the change of the integration variable $u\mapsto-u$. 
Because of the singularity of the integrand in \eqref{eq:J} at $u=0$, the expression of $(J_a f)(t)$ in \eqref{eq:symm} will usually be more convenient in computation than the expression of $(J_a f)(t)$ in \eqref{eq:J}. 
By the same change of the integration variable, one has the following parity property of the transformation $J_a$: 
\begin{equation}\label{eq:par}
	(J_a f)(-t)=-(J_{-a}f^-)(t), 
\end{equation}
where $f^-(u):=f(-u)$ for all real $u$. Of course, if $f$ is a c.f., then $f^-=\bar f$, where, as usual, the horizontal bar 
above the symbol denotes the complex conjugation.  

Clearly, one may also write 
\begin{equation*}
	(J_a f)(t)=\frac1{2\pi i}\,\int_{-\infty}^\infty [e^{-iua}f(t+u)-f(t)]\,\frac{du}u. 
\end{equation*}
Moreover, in place of $\frac1{4\pi i}\,\int_{-\infty}^\infty$ in \eqref{eq:symm} one may write 
$\frac1{2\pi i}\,\int_0^\infty$, where $\int_0^\infty:=\lim_{\vp\downarrow0,\,A\uparrow\infty}\int_\vp^A$.

We shall now see that the integral in \eqref{eq:J} always exists -- as long as $f$ is a c.f.; moreover, the value of $(J_a f)(t)$ is no greater than $\frac12$ in modulus. 
In fact, one has the following basic proposition. 

\begin{proposition}\label{prop:J}\ One has 
\begin{equation}\label{eq:J=}
	(J_a f)(t)=\tfrac12\,\E e^{itX}\sign(X-a). 
\end{equation}
	Moreover, for all r.v.'s $X$, all $a$, and all $\vp$ and $A$ such that $0<\vp<A<\infty$  
\begin{equation}\label{eq:J<}
\big|(J_{a;\vp,A} f)(t)\big|<1. 	
\end{equation}
\end{proposition}

\begin{proof}[Proof of Proposition~\ref{prop:J}]
Consider first the (extreme) case when the real-valued r.v.\ $X$ takes only one value, say $x$. 
Then $f(t)=e^{itx}$ and hence 
\begin{align*}
	(J_a f)(t)=&\frac{e^{itx}}{2\pi i}\,\int_{-\infty}^\infty e^{iu(x-a)}\,\frac{du}u 
	=\frac{e^{itx}}{2\pi}\,\int_{-\infty}^\infty\frac{\sin u(x-a)}u\,du
	=\frac12\,e^{itx}\sign(x-a),  
\end{align*}
so that \eqref{eq:J=} holds in this case. Similarly, still in this same extreme case, 
\begin{align*}
	2\pi\,|(J_{a;\vp,A}f)(t)|=&\Big|\int_{\vp,A}\frac{\sin u(x-a)}u\,du\Big|
	\le\int_{-\pi}^\pi\frac{\sin v}v\,dv<\int_{-\pi}^\pi\,dv=2\pi,   
\end{align*}
so that \eqref{eq:J<} holds in this case as well. 
It remains to use Fubini's theorem to obtain \eqref{eq:J=} and \eqref{eq:J<} in general. 
\end{proof}

Consider the r.v. 
\begin{equation*}
	X_+:=0\vee X=\max(0,X), 
\end{equation*}
the positive part of the r.v.\ $X$. In view of the obvious identity 
\begin{equation}\label{eq:^{itX_+}=}
	2e^{itX_+}=1+e^{itX}+e^{itX}\sign X-\sign X, 
\end{equation}
one immediately obtains the following corollary of Proposition~\ref{prop:J}, which gives an expression of the c.f.\ of $X_+$ in terms of the c.f.\ of $X$. 

\begin{corollary}\label{cor:X_+}\  
\begin{equation}\label{eq:X_+,c.f.}
	\E e^{itX_+}=\tfrac12\,[1+f(t)]+(Jf)(t)-(Jf)(0).  
\end{equation}
\end{corollary} 
Here and in what follows, we set 
\begin{equation}\label{eq:J=J_0}
	J:=J_0, 
\end{equation}
for brevity. 
One may note that 
\begin{equation*}
	J=\tfrac i2\,H, 
\end{equation*}
where $H$ is the Hilbert transform, 
given by the formula 
\begin{equation*}
	(Hf)(t):=\frac1\pi\,\int_{-\infty}^\infty\frac{f(s)\,ds}{t-s}
	=-\frac1\pi\,\int_0^\infty[f(t+u)-f(t-u)]\,\frac{du}u. 
\end{equation*}
A comprehensive treatment of theoretical properties, applications, and special instances of the Hilbert transform is presented in \cite{king_hilbert1,king_hilbert2}. 

\begin{remark}\label{rem:2 to 1}
The difference $(Jf)(t)-(Jf)(0)$ in \eqref{eq:X_+,c.f.} can be rewritten as one integral: 
\begin{equation*}
	2\pi i\,[(Jf)(t)-(Jf)(0)]=\int_{-\infty}^\infty[f(t+u)-f(u)]\,\frac{du}u, 
\end{equation*}
of course understood in general in the principal-value sense. 
\end{remark}

Formulas such as \eqref{eq:X_+,c.f.} will be especially useful when the c.f.\ $f$ of $X$ can be expressed more readily than (say) the density of $X$ -- which will usually be the case when e.g.\ $X$ is the sum of independent r.v.'s; cf. e.g.\ \cite{nonunif,pin-hoeff-published}. 
Also, because both sides of identities such as \eqref{eq:X_+,c.f.} are affine in $f$ and hence in the distribution of $X$, these identities will be suitable when the distribution of $X$ is presented as the mixture of simpler distributions. 

Another kind of application of identity \eqref{eq:X_+,c.f.} will be given in 
Subsection~\ref{spitzer} of the present note, where a more explicit and compact form of Spitzer's identity is obtained. 

The following corollary is a generalization of Corollary~\ref{cor:X_+}, providing an expression of the joint c.f.\ of $X$, $X_+$, and $X_-:=(-X)_+$ in terms of $f$. 

\begin{corollary}\label{cor:X,X_+,X_-}\  
\begin{equation}\label{eq:X,X_+,X_-}
	\E e^{i(\al X+\be X_++\ga X_-)}=\tfrac12\,[f(\al+\be)+f(\al-\ga)]+(Jf)(\al+\be)-(Jf)(\al-\ga).  
\end{equation}
\end{corollary} 

\begin{proof}[Proof of Corollary~\ref{cor:X,X_+,X_-}]
Note that $\al X+\be X_++\ga X_-=(\al-\ga)X+(\be+\ga)X_+$. 
Now multiply identity \eqref{eq:^{itX_+}=} term-wise by $e^{isX}$, then use the identity $e^{isX}e^{itX}=e^{i(s+t)X}$, then replace $s$ and $t$ respectively by $\al-\ga$ and $\be+\ga$, and finally use \eqref{eq:J=} and \eqref{eq:J=J_0}. 
This will yield \eqref{eq:X,X_+,X_-}. 
\end{proof}

With $\al=0$ and $\be=\ga=t$, identities \eqref{eq:X,X_+,X_-} and \eqref{eq:par} allow one to immediately express the c.f.\ of the r.v.\ $|X|=X_++X_-$ in terms 
of $f$: 
\begin{equation}\label{eq:=H,abs}
	\E e^{it|X|}
	=\Re f(t)+2\,(J\Re f)(t)=\Re f(t)+2\,\Re(Jf)(t).  
\end{equation}

Take now any real $a$ and $b$ such that $a\le b$ and introduce the r.v. 
\begin{equation*}
	X_{a,b}:=a\vee(b\wedge X),  
\end{equation*}
so that all the values of $X_{a,b}$ are in the interval $[a,b]$. 
Then, 
using the identity 
\begin{equation*}
	2e^{itX_{a,b}}=e^{i t a}+e^{i t b}+(e^{i t X}-e^{i t a})\sign(X-a)+(e^{i t b}-e^{i t X})\sign(X-b), 
\end{equation*}
similarly to Corollary~\ref{cor:X_+} one obtains 
the following expression of the c.f.\ of $X_{a,b}$ in terms of the c.f.\ of $X$. 

\begin{corollary}\label{cor:X_ab}\  
\begin{equation*}
	\E e^{itX_{a,b}}=\tfrac12\,(e^{i t a}+e^{i t b})+(J_a f)(t)-e^{i t a}(J_a f)(0)+e^{i t b}(J_b f)(0)-(J_b f)(t). 
\end{equation*}
\end{corollary}
A remark similar to Remark~\ref{rem:2 to 1} can of course be made concerning Corollary~\ref{cor:X_ab}. 


\section{Applications}\label{appls}

\subsection{Joint c.f.\ of the prices of the stock, call option, and put option}
In finance, $C:=(S-K)_+$ and $P:=(K-S)_+=(S-K)_-$ are the respective values of a call option and a put option with strike price $K$ when the price of the underlying stock is $S$. It follows immediately from \eqref{eq:X,X_+,X_-} that the joint c.f.\ in the title of this subsection is given by the formula
\begin{equation*}
	\E e^{i(\al S+\be C+\ga P)} 
	=e^{itK}\,
	\big\{\tfrac12\,[f_K(\al+\be)+f_K(\al-\ga)]+(Jf_K)(\al+\be)-(Jf_K)(\al-\ga)\big\},   
\end{equation*} 
where $f_K$ is the c.f.\ of the r.v.\ $S-K$, so that $f_K(t)=e^{-itK}\E e^{itS}$. 

\subsection{Random walks and Spitzer's identity}\label{spitzer}
Let $X,X_1,X_2,\dots$ be any independent identically distributed (iid) r.v.'s. For all $n=0,1,\dots$, let $S_n:=\sum_{j=1}^n X_j$; in particular, by the standard convention about the sum of an empty family, $S_0=0$; let also 
\begin{equation*}
	M_n:=\max_{0\le k\le n}S_k 
\end{equation*}
and 
\begin{equation*}
	\vpi_n(s,t):=\E\exp\{i[sM_n+t(M_n-S_n)]\}, 
\end{equation*}
which latter determines the joint c.f.\ of the pair $(M_n,S_n)$ -- cf.\ formula~(6.2) in \cite{spitzer}. 

Consider also the Markov process $(T_n)$ defined recursively by the conditions  
\begin{equation}\label{eq:T_n}
	T_0=0\quad\text{and}\quad T_n=(T_{n-1}+X_n)_+\text{ for $n=1,2,\dots$.} 
\end{equation}
It is well known that formula \eqref{eq:T_n} describes a large class of 
queuing systems; 
see e.g.\ 
\cite[Ch.\ IV, Section~2, formula~(1)]{borovk_queus}.  
It is also well known (see e.g.\ page~155 in \cite{spitzer_TAMS60}) that $M_n$ has the same distribution as $T_n$, for each $n=0,1,\dots$. 

Therefore, letting $f_n$ stand for the c.f.\ of $M_n$, by \eqref{eq:X_+,c.f.} one has the following recurrence: 
\begin{equation*}\label{eq:f_n}
	f_0=1\quad\text{and}\quad f_n=\tfrac12\,(1+f_{n-1}f)+J(f_{n-1}f)-J(f_{n-1}f)(0)\text{ for $n=1,2,\dots$.}  
\end{equation*}

Similarly, take any real $a$, $b$, $x$ such that $a\le x\le b$ and 
consider the Markov process $(U_n)=(U_{a,b,x;n})$ defined recursively by the conditions  
\begin{equation*}
	U_0=x\quad\text{and}\quad U_n=a\vee\big(b\wedge(U_{n-1}+X_n)\big)\text{ for $n=1,2,\dots$.} 
\end{equation*}
Thus, $(U_n)$ is a random walk on the interval $[a,b]$, with barriers at $a$ and $b$. 
Then, letting $g_n=g_{a,b,x;n}$ stand for the c.f.\ of $U_n$, by Corollary~\ref{cor:X_ab} one has the following recurrence: $g_0=e^{ix\cdot}$ and, for $n=1,2,\dots$, 
\begin{equation*}\label{eq:g_n}
	g_n=
	\tfrac12\,(e^{ia\cdot}+e^{ib\cdot})+J_a(g_{n-1}f)-e^{ia\cdot}J_a(g_{n-1}f)(0)+e^{ib\cdot}J_b(g_{n-1}f)(0)-J_b(g_{n-1}f). 
\end{equation*}

\begin{center}
	***
\end{center}

A general form of the famous Spitzer identity was given in Theorem~6.1 in \cite{spitzer}:  
\begin{equation}\label{eq:spitzer}
	\sum_{n=0}^\infty\vpi_n(s,t)z^n=\exp\sum_{k=1}^\infty\frac{z^k}k\,[\psi_k(s)+\th_k(t)-1]; 
\end{equation}
here and in what follows, $z$ is any complex number with $|z|<1$ and 
\begin{equation}\label{eq:psi,th}
	\psi_k(s):=\E\exp\{is(S_k)_+\}\quad\text{and}\quad\th_k(t):=\E\exp\{it(S_k)_-\}. 
\end{equation}
Spitzer's identity \eqref{eq:spitzer} and especially its special case with $t=0$ have been used in large number of papers, mainly for asymptotic analysis using Tauberian types theorems. 

Spitzer's identity reduces the power-series transform of the ``difficult'' joint c.f.'s $\vpi_n$ of the pairs $(M_n,M_n-S_n)$
in the left-hand side of \eqref{eq:spitzer} to a power-series transform of ``easier'' c.f.'s $\psi_k$ and $\th_k$ of $(S_k)_+$ and $(S_k)_-$, the positive and negative parts of the sums $S_k=\sum_{j=1}^k X_j$ of the iid r.v.'s $X_j$. 


Using Corollary~\ref{cor:X_+} of the present note, we shall provide more explicit and compact integral expressions of the 
left-hand side of \eqref{eq:spitzer} in terms of just the c.f.\ $f$ of the r.v.\ $X$ (which was assumed to be identical in distribution to each $X_j$). 
Indeed, by \eqref{eq:X_+,c.f.}, 
\begin{equation}\label{eq:psi}
	\psi_k(s)=\tfrac12\,[1+f(s)^k]+(Jf^k)(s)-(Jf^k)(0), 
\end{equation}
where, as before, $f$ is the c.f.\ of $X$, and $Jf^k:=J(f^k)$. 
Next,
\begin{align}
	\sum_{k=1}^\infty\frac{z^k}k\,(Jf^k)(s)=&\sum_{k=1}^\infty\frac{z^k}k\,\lim_{\vp,A}(J_{\vp,A}f^k)(s) \label{eq:sum lim} \\ 
	=&\lim_{\vp,A}\sum_{k=1}^\infty\frac{z^k}k\,(J_{\vp,A}f^k)(s) \label{eq:lim sum} \\ 
	=&\lim_{\vp,A}\Big(J_{\vp,A}\sum_{k=1}^\infty\frac{z^k}k\,f^k\Big)(s)
	=\Big(J\ln\frac1{1-zf}\Big)(s); \label{eq:lim J sum} 
\end{align}
here $\lim_{\vp,A}:=\lim_{\vp\downarrow0,A\uparrow\infty}$ and $J_{\vp,A}:=J_{0;\vp,A}$. 
The equality in \eqref{eq:sum lim} and the second equality in \eqref{eq:lim J sum} follow immediately by the definition \eqref{eq:J=J_0} of $J=J_0$ and formula \eqref{eq:J_a=} for $J_a$. 
The equality in \eqref{eq:lim sum} follows by dominated convergence, in view of \eqref{eq:J<} and the condition $|z|<1$. 
Finally, in view of definitions $J_{\vp,A}:=J_{0;\vp,A}$ and \eqref{eq:J_{a;vp,A}}, the first equality in \eqref{eq:lim J sum} follows by the linearity of the operator $J_{\vp,A}$ and another instance of dominated convergence, because for each pair $(\vp,A)\in(0,\infty)^2$ with $\vp<A$ and all $u\in[-A,-\vp]\cup[\vp,A]$ one has 
\begin{equation*}
	\Big|\frac1u\,\sum_{k=1}^n\frac{z^k}k\,f(u)^k\Big|\le\frac1\vp\,\sum_{k=1}^n\frac{|z|^k}k
	\le\frac1\vp\,\ln\frac1{1-|z|}, 
\end{equation*}
for all natural $n$. 

The second equality in \eqref{eq:lim J sum} requires the following comment: Whereas the original definition of $J_a f$ in \eqref{eq:J} was stated only for the case when $f$ is a c.f., it obviously remains valid for any continuous function $f\colon\R\to\CC$, as long as the limit in \eqref{eq:J_a=} exists. In view of \eqref{eq:J=J_0}, this clarifies the meaning of the expression $\big(J\ln\frac1{1-zf}\big)(s)$ in \eqref{eq:lim J sum}; moreover, it follows from the proof given in the previous paragraph that the limit $\big(J\ln\frac1{1-zf}\big)(s)$ exists (and is finite).   

By \eqref{eq:psi} and \eqref{eq:sum lim}--\eqref{eq:lim J sum}, 
\begin{equation}\label{eq:sum psi}
	\sum_{k=1}^\infty\frac{z^k}k\,[\psi_k(s)-\tfrac12]
	=\frac12\,\ln\frac1{1-zf(s)}+\Big(J\ln\frac1{1-zf}\Big)(s)-\Big(J\ln\frac1{1-zf}\Big)(0). 
\end{equation}
By \eqref{eq:psi,th}, identity \eqref{eq:sum psi} will hold if $\psi_k$ and $f$ are replaced there by $\th_k$ and $\bar f$, respectively. So, recalling the definitions \eqref{eq:J=J_0} and \eqref{eq:J} of $J=J_0$ and $J_a$, one sees that the right-hand side of Spitzer's identity \eqref{eq:spitzer} is equal to \break   
$[(1-zf(s))
(1-z\bar f(t))]^{-1/2}\exp\J$, where  
\begin{equation}\label{eq:JJ:=}
	\J:=\frac1{2\pi i}\,\int_{-\infty}^\infty\frac{du}u\,\Big(\ln\frac{1-zf(u)}{1-zf(s+u)}
	+\ln\frac{1-z\bar f(u)}{1-z\bar f(t+u)}\Big). 
\end{equation}
Here we used the identities $\ln\frac1{1-zf(s+u)}-\ln\frac1{1-zf(u)}=\ln\frac{1-zf(u)}{1-zf(s+u)}$ and \break 
$\ln\frac1{1-z\bar f(t+u)}-\ln\frac1{1-z\bar f(u)}=\ln\frac{1-z\bar f(u)}{1-z\bar f(t+u)}$, which are true because \break 
$\arg(1-zf(s+u))\in(-\frac\pi2,\frac\pi2)$ and hence $\arg\frac{1-zf(u)}{1-zf(s+u)}\in(-\pi,\pi)$ for any c.f.\ $f$, given that $|z|<1$ and $s$ and $u$ are in $\R$. However, in general one may not be able to combine the two logarithms in \eqref{eq:JJ:=} into one, since the equality $\ln z_1+\ln z_2=\ln(z_1z_2)$ fails to hold for some nonzero complex $z_1$ and $z_2$. 

Using again the change  of the integration variable $u\mapsto-u$ (cf.\ \eqref{eq:symm}), one can finally rewrite \eqref{eq:spitzer} as follows. 

\begin{corollary}\label{cor:spitzer}
\begin{equation}\label{eq:spitzer-rewr}
\begin{multlined}
	\sum_{n=0}^\infty\vpi_n(s,t)z^n=
	\frac1{\sqrt{(1-zf(s))(1-zf(-t))}} \\ 
	\,\exp\Big\{
	\frac1{4\pi i}\,\int_{-\infty}^\infty\frac{du}u\,\Big(
	\ln\frac{1-zf(s-u)}{1-zf(s+u)}
	+\ln\frac{1-zf(-t+u)}{1-zf(-t-u)}\Big)
	\Big\}. 
\end{multlined}	
\end{equation} 
In the special case $t=0$, \eqref{eq:spitzer-rewr} yields 
\begin{equation}\label{eq:spitzer-rewr,M}
\begin{multlined}
	\sum_{n=0}^\infty\E e^{isM_n}z^n=
	\frac1{\sqrt{(1-zf(s))(1-z)}} \\ 
	\,\exp\Big\{
	\frac1{4\pi i}\,\int_{-\infty}^\infty\frac{du}u\,\Big(
	\ln\frac{1-zf(s-u)}{1-zf(s+u)}
	+\ln\frac{1-zf(u)}{1-zf(-u)}\Big)
	\Big\}. 
\end{multlined}	
\end{equation} 
\end{corollary}

Identities somewhat similar to \eqref{eq:spitzer-rewr} and \eqref{eq:spitzer-rewr,M} were obtained by \cite{spitzer_duke57,spitzer_TAMS60} and \cite{borovk_SMZh62,borovk_TVP70}. In particular, Theorem~3 in \cite{spitzer_duke57} expresses the left-hand side of \eqref{eq:spitzer-rewr,M} in terms of (a double integral of a rational function of) $f$ -- under the condition the function $t\mapsto(1-f(t))/t$ be integrable in a neighborhood of $0$. As mentioned in \cite{spitzer_duke57}, earlier a result close to Theorem~3 in \cite{spitzer_duke57} was obtained by \cite{pollaczek} (under the condition that $f$ is analytic in a neighborhood of $0$). The method in \cite{spitzer_duke57} was based on a general inversion formula due to \cite{hewitt53}. 
Developing these results further, 
Theorems~5.2 A and 5.2 B in \cite{spitzer_TAMS60} present, in the case when the distribution of $X$ is symmetric and either continuous or lattice, an expression in terms of the c.f.\ $f$ of a limit of the left-hand side of Spitzer's identity as $z\uparrow1$, but for the Laplace transform (rather than the Fourier one).  
In the case when $\E|X|<\infty$ and the distribution of $X$ is either lattice or with a nontrivial absolutely continuous component, 
formula (36) in \cite{borovk_TVP70} formula presents an expression in terms of $f$ for the ``positive'' component $w_+$ of the so-called $V$-factorization of the function $1-f$ and hence for the Laplace--Fourier transform of the distribution of the r.v.\ $M_\infty:=\lim_{n\to\infty}M_n$, provided that $M_\infty<\infty$ almost surely. 
In comparison, recall here that identities \eqref{eq:spitzer-rewr} and \eqref{eq:spitzer-rewr,M} hold without any restrictions on the distribution of $X$; also, the methods presented here differ from those in the mentioned papers.  

\bibliographystyle{abbrv}

\bibliography{C:/Users/ipinelis/Dropbox/mtu/bib_files/citations12.13.12}


\end{document}